\documentclass[12pt]{article}

\usepackage[utf8]{inputenc}
\usepackage[T1]{fontenc}
\usepackage[a4paper, left=2cm, right=2cm, top=2.5cm, bottom=2cm]{geometry}
\usepackage{amsfonts} 
\usepackage{amsmath} 
\usepackage{amssymb} 
\usepackage{amsthm} 
\usepackage[numbers]{natbib}
\usepackage{babel}
\usepackage{hyperref} 
\usepackage{cleveref} 
\usepackage{color}
\usepackage{dsfont} 
\usepackage{enumitem}
\usepackage{graphicx}
\usepackage{pifont}
\usepackage{tabto}
\usepackage[all]{xy}
\usepackage{nomencl}
\usepackage{wrapfig}
\usepackage{authblk}

\usepackage{tikz}
\usepackage{tkz-euclide}
\usepackage{pgfplots}
\usetikzlibrary{positioning, fit, calc, shapes, arrows, matrix, shapes.misc, intersections}

\theoremstyle{plain}
\newtheorem{thm}{Theorem}[section]
\newtheorem{lem}[thm]{Lemma}

\theoremstyle{definition}
\newtheorem{dfnt}{Definition}[section]

\theoremstyle{remark}
\newtheorem*{rmq}{Remark}

\crefname{thm}{theorem}{theorems}
\crefname{lem}{lemme}{lemmes}
\crefname{prop}{proposition}{propositions}
\crefname{cor}{corollaire}{corollaires}
\crefname{dfnt}{definition}{definitions}
\crefname{exmp}{exemple}{exemples}
\crefname{xca}{exercice}{exercices}
\crefname{rmq}{remarque}{remarques}
\crefname{note}{note}{notes}
\crefname{case}{case}{cases}
\crefname{hyp}{hypothèse}{hypothèse}

\newcommand{\FF}{\mathcal{F}}
\newcommand{\MM}{\mathcal{M}}
\newcommand{\QQ}{\mathcal{Q}}

\newcommand{\PP}{\mathcal{P}}

\begin{document}

\title{Polynomial rate of mixing for a family of billiard flows}
\author[1]{Bonnafoux, Etienne \\ email: \href{mailto:etienne.bonnafoux@polytechnique.edu}{etienne.bonnafoux@polytechnique.edu}}
\affil[1]{Centre de mathématiques Laurent-Schwartz, Ecole polytechnique, 91128 Palaiseau Cedex, France}

\providecommand{\keywords}[1]{\textbf{\textit{Key words:}} #1}
\providecommand{\classification}[1]{\textbf{\textit{2020 Mathematics Subject Classification:}} #1}

\maketitle

\begin{abstract}
  We prove that the continuous correlation function decrease polynomially for two families of billiard studied by Chernov and Zhang \cite{chernov2005family}. The main computation is to show that the return function is Hölder on stable and unstable manifold.
\end{abstract}

\section{Introduction}

A planar billiard table $\QQ$ is a connected compact part of $\mathbb{R}^2$. This set generates dynamical systems as one consider the trajectory of a particle going straight in the interior and bouncing off its boundary $\partial \QQ$.

Here we will focus on two families of billiard table studied by Chernov and Zhang \cite{chernov2005family} with a nonuniformly hyperbolic boundary. For $\beta>2$, let $$
g_{\beta}(x)= |x|^{\beta}+1.
$$

Let $\QQ$ be a closed compact domain of $\mathbb{R}^2$ bounded by
\begin{itemize}
  \item The curve of $-g_{\beta}$
  \item The curve of $g_{\beta}$ or the horizontal axis
  \item several strictly convex curves with nowhere vanishing curvature and no cusp.
\end{itemize}

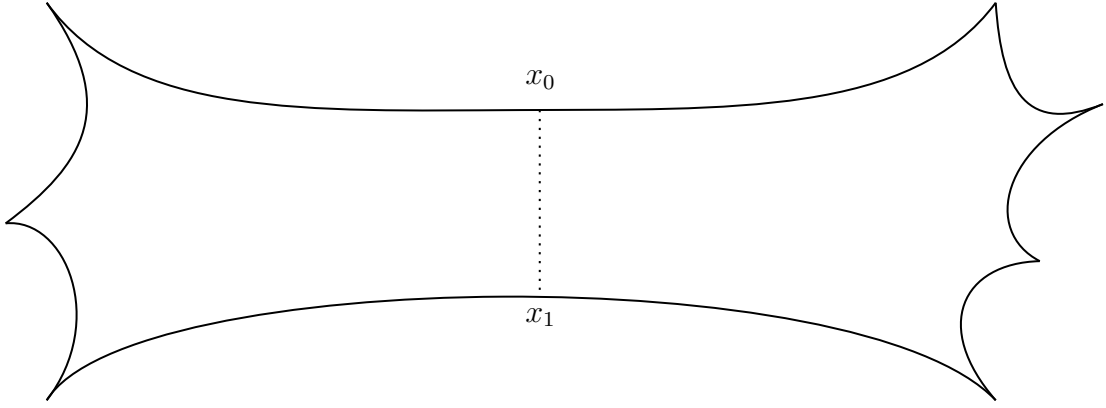
\begin{figure}[!h]
\centering

\tikzset{every picture/.style={line width=0.75pt}} 

\tikzset{every picture/.style={line width=0.75pt}} 

\begin{tikzpicture}[x=0.75pt,y=0.75pt,yscale=-1,xscale=1]

\draw    (80,255) .. controls (100.32,221.39) and (202.31,203.38) .. (307.5,202.87) .. controls (411.98,202.36) and (519.61,219.12) .. (553.5,255) ;
\draw    (80,55) .. controls (121.5,113) and (210.31,109.51) .. (315.5,109) .. controls (420.69,108.49) and (511.5,112) .. (553.5,55) ;
\draw    (553.5,55) .. controls (556.5,101) and (569.5,121) .. (607,106) ;
\draw    (607,106) .. controls (556.5,126) and (546.5,170) .. (575.5,185) ;
\draw    (575.5,185) .. controls (536.5,186) and (521.5,218) .. (553.5,255) ;
\draw    (59.5,166) .. controls (99.5,136) and (116.5,107) .. (80,55) ;
\draw    (80,255) .. controls (110.5,215) and (89.5,163) .. (59.5,166) ;
\draw  [dash pattern={on 0.84pt off 2.51pt}]  (326,109) -- (326,203) ;

\draw (317.33,87.73) node [anchor=north west][inner sep=0.75pt]    {$x_{0}$};
\draw (317.33,207.07) node [anchor=north west][inner sep=0.75pt]    {$x_{1}$};

\end{tikzpicture}

\caption{Example of consider billard, $x_0$ and $x_1$ are point where the curvature vanish.}
\end{figure}

\begin{rmq}
  In the second bullet point, one pass from the two models by symmetrizing along the horizontal axis or by folding along it.
\end{rmq}

For the above billiard tables we call $\MM$ the standard cross-section $\partial \QQ \times [-\pi/2 , \pi /2]$ parametrized by two parameters $(r,\phi)$ where $r$ is the arc length of the boundary (taking an implicit origin for each arc) and $\phi$ the angle with the normal vector to the boundary.

We call $\{\FF^t\}$ the induced billiard flow and $T$ the first return map to $\MM$.
This map and this flow preserve two continuous measure that we denote by $\mu_{\MM}$ for the measure on $\MM$ and $\nu_{\QQ \times \mathbb{S}^1}$ on $\QQ \times \mathbb{S}^1$.

Let $f,g \in L^2(\MM,\mu_{\MM})$ be two functions. The discrete correlation function between the two is defined as $$
C_n(f,g,T,\mu_{\MM}) = \left\vert \int_{\MM} (f \circ T^{n})g d \mu_{\MM} - \int_{\MM} f d \mu_{\MM}\int_{\MM}g d \mu_{\MM} \right\vert.
$$
Similarly, the continuous correlation is also defined for $f,g \in L^2(\QQ \times \mathbb{S}^1,\nu_{\QQ \times \mathbb{S}^1})$as
$$
C_t(f,g,\FF,\nu_{\QQ \times \mathbb{S}^1}) = \left\vert \int_{\QQ \times \mathbb{S}^1} (f \circ \FF^{t})g d \nu_{\QQ \times \mathbb{S}^1} - \int_{\QQ \times \mathbb{S}^1} f d \nu_{\QQ \times \mathbb{S}^1} \int_{\QQ \times \mathbb{S}^1}g d \nu_{\QQ \times \mathbb{S}^1} \right\vert.
$$

Chernov and Zhang \cite{chernov2005family} showed that for this family of billiard,
\begin{thm}
  For this family of billiard tables, the correlations for the billiard map $T :\MM \to \MM$ and piecewise Hölder continuous function $f,g$ on $\MM$ decay as $$
  |C_n(f,g,T,\mu_{\MM})| \leq C_{Th_{1.1}} \frac{(\ln n)^{a+1}}{n^a}
  $$
  where $a=\frac{\beta+2}{\beta-2}$, and $C_{Th_{1.1}}$ a constant.
\end{thm}

The goal of this paper is to demonstrate a similar result for the continuous correlation, that is

\begin{thm}
  \label{thm_vitesse_billiard_continue}
  For this family of billiard tables, the correlations for the billiard flow $\mathcal{F}$ and for some function $f,g$ on $\QQ \times \mathbb{S}^1$, with some regularity, decay as $$
  |C_n(f,g,\FF,\nu_{\QQ \times \mathbb{S}^1})| \leq C_{Th_{1.2}} \|f \| \| g \| \frac{1}{t^{a}}
  $$
  where $a=\frac{\beta+2}{\beta-2}$, and $C_{Th_{1.2}}$ a constant.
\end{thm}

We will detail later the regularity needed and the norm involved in Theorem \ref{thm_vitesse_billiard_continue}. Nevertheless, we denote $d_{\MM}(\cdot,\cdot)$ the natural distance on $\MM$ and $d_{\QQ \times \mathbb{S}^1}(\cdot,\cdot)$ the one on $\QQ \times \mathbb{S}^1$.

\section{Theoretical aspects of the proof}

To demonstrate \ref{thm_vitesse_billiard_continue} we will apply an abstract framework. More precisely, we will show that this flow is a so called Gibbs-Markov flow, with polynomially decreasing return time tail and without approximate eigenfunctions.
We will explain these terms in the next subsections. For now, let's state the theorem that will be used to demonstrate Theorem \ref{thm_vitesse_billiard_continue}.

\begin{thm}[Corollary 8.1 of \cite{Balint2019polynomial}]
  \label{thm_theorique_Gibbs_Markov_map}
  Let $T_t: M \to M$ be a nonuniformly hyperbolic flow on a metric space $(M,d_M)$ with $diam M < 1$.Fix $ \eta \in (0,1]$. Let $X \subset M$ be a Borel subset, with $f : X \to X$ the return map to $X$ and $h$ the roof function. We suppose that $0< \inf h$, $\sup h < + \infty$ and $h$ is Hölder.

   We suppose that $f$ is modelled by a Young tower with a horseshoe set $Y \subset X$, with an invariant ergodic measure $\mu$. $\bar{Y}$ is the quotient of $Y$ by stable leaves of the dynamic and $\phi$ is the return time to $Y$. If $\tau$ is the number of step made in the Young tower then $$
\phi(x) = \sum_{k=0}^{\tau(x)-1} h(f^k(x))
   $$

   We assume that condition (H) is satisfied, we assume that $\mu(\{ y \in \bar{Y}, \phi(y)>t \}) = O(t^{-\delta})$ for some $\delta > 1$ and we assume absence of approximate eigenfunctions. Then there exists $m \geq 1$ and $ C_{Th_{2.1}} > 0$ such that $$
  |\rho_{v,w}(t)| \leq C_{Th_{2.1}} (\| v \|_{\eta} + \| v \|_{0,\eta})\| w \|_{m,\eta})t^{-\delta+1}
  $$
  for all $v \in \mathcal{C}_{\eta}(M) \cap \mathcal{C}_{0,\eta}(M)$, $w \in \mathcal{C}_{m,\eta}(M)$, $t>1$.
\end{thm}

The next subsections explain the different terms appearing in the hypothesis of Theorem \ref{thm_theorique_Gibbs_Markov_map} and, for some hypothesis, demonstrate them on the fly, with formal arguments. We will end this section by indicating which computations remain to finish the proof. The next section will deal with these computations.

So for the rest of the section we will discuss
\begin{itemize}
  \item Young tower and condition (H),
  \item polynomial tail for the return function,
  \item absence of approximate eigenfunctions,
  \item and the relation with norms appearing in Theorem \ref{thm_theorique_Gibbs_Markov_map} and Hölder observables.
\end{itemize}

\subsection{Young tower and condition (H)}

A \emph{Gibbs-Markov flow} is an abstract framework developed by Bálint, Butterley and Melbourne \cite{Balint2019polynomial}. They used it to compute rate of mixing for some families of flow. In the same work, they indicated how to link this class of flow to suspensions over \emph{Young tower}.

Young tower is a method developed by Young to understand rate of mixing of billiard. First she introduced it for exponential mixing and then expend the range of application to slowly mixing systems.

These two notions are linked in Section 7.2 of \cite{Balint2019polynomial}. The authors indicated that if a flow is non-uniformly hyperbolic and modelled by a Young tower then if two conditions called condition (6.1) and condition (H) are checked then, the flow is Gibbs-Markov.

In our case, the work of Chernov and Zhang \cite{chernov2005billiards}, indicates that the billiard flow of these tables are, indeed, modelled by a Young tower with a hyperbolic part described below, in Section \ref{sec_calcul_billiard}.

Furthermore, Lemma 8.3 of the article of Bálint, Butterley and Melbourne \cite{Balint2019polynomial} reduce the  condition (H) to Hölder estimations for the roof function $\phi$ on stable and unstable leaves. These estimates are given with respect to a metric given by the partition $\{ Y_j \}$ given in the definition of Gibbs-Markov flow.

Specifically, they ask that if $y,y' \in Y$,
\begin{itemize}
  \item $| \phi(y) - \phi(y')| \leq C_4' d_M(y,y')$ for all $y' \in W^s(y)$,
  \item $| \phi(y) - \phi(y')| \leq C_4' \iota^{s(y,y')}$ for all $y' \in W^u(y),s(y,y') \geq 1$,
\end{itemize}
with $s(\cdot,\cdot)$ the separation time in the Young tower and $\iota \in (0,1)$.

Remark that we can relax the Lipschitz condition to a Hölder condition by considering the distance $d_M'(\cdot,\cdot)=d_M(\cdot,\cdot)^{\gamma}$. If we prove a $\gamma$-Hölder inequality for $d_M$, we have a Lipschitz inequality for $d_M'$.

In our case we will compute Hölder estimations with the natural distance on the border cross-section of flow of the billiard. These estimations are stronger than the of Lemma 8.3 of \cite{Balint2019polynomial} given above.

Namely, we will show that for a hyperbolic part $\Sigma$ of the cross-section $\partial \QQ \times \mathbb{S}^1$, for $y,y' \in C^n$, where $C^n$ are some neighborhood described below in Section \ref{sec_calcul_billiard},
\begin{itemize}
  \item $|\Theta(y)-\Theta(y')| \leq C_4 d_{\MM}(y,y')^{\gamma}$ if $y \in W^s(y')$,
  \item $|\Theta(y)-\Theta(y')| \leq C_4 d_{\MM}(T_{\Sigma}(y),T_{\Sigma}(y'))^{\gamma}$ if $y \in W^u(y')$,
\end{itemize}
where $T_{\Sigma}$ is the return map to $\Sigma$ and $\Theta$ the associated roof function and $\gamma \in (0,1)$.

\subsection{Polynomial tail for the return function}

A first condition to apply Theorem \ref{thm_theorique_Gibbs_Markov_map}, is that the tail of the roof function is polynomial that is $$\mu(\{ y \in \bar{Y}, \phi(y) > t \}) = O(t^{-\delta})$$ for some $\delta > 1$.

In our case, for the hyperbolic past $\Sigma$, Chernov and Zhang have defined the function $R : \Sigma \to \mathbb{N}$, which counts the number of bounces of a trajectory on the border of the billiard before it goes back to the hyperbolic part $\Sigma$. They have shown that \cite{chernov2005family}
$$
\mu_{\MM|_{\Sigma}}(\{ x \in \Sigma, R(x) > n\}) \leq C_5 n^{-a-1}
$$
with $a=\frac{\beta+2}{\beta-2}$. As the diameter of the billiard table $D$ is finite, we have for our roof function $\Theta$
\begin{align*}
  \mu_{\MM|_{\Sigma}}(\{ x \in \Sigma , \Theta(x) > t \}) & \leq \nu(\{ x \in \Sigma , R(x) > \lfloor t/D \rfloor \}) \\
  & \leq C_5 \lfloor t/D \rfloor^{-a-1} \\
  & \leq 2 C_5 (t/D)^{-a-1}
\end{align*}
with the last inequality true for $t$ large enough. So this condition is verified in our case.

\subsection{Absence of approximate eigenfunctions}

A second condition in Theorem \ref{thm_theorique_Gibbs_Markov_map} that we should check is the absence of approximate eigenfunctions. We will recall quickly the definition and then give indication on how to prove it.

For $b\in \mathbb{R}$, we define the operators $M_b : L^{\infty}(\bar{Y}) \to L^{\infty}(\bar{Y})$ by
$$
  M_b v = e^{i b \phi}v \circ \bar{F}
$$
\begin{dfnt}[approximate eigenfunctions]
  A subset $Z_0 \subset \bar{Y}$ is a \emph{finite subsystem} of $\bar{Y}$ if $Z_0 =\cap_{n \geq 0} \bar{F}^{-n}Z$ where $Z$ is the union of finitely many elements from the partition $\{ \bar{Y}_j \}$.
   We say that $M_b$ has \emph{approximate eigenfunctions} on $Z_0$ if for any $\alpha_0 > 0$, there exist constants $\alpha,\zeta > \alpha_0$ and $C>0$, and sequences $|b_k| \to \infty$, $\psi_k \in [0,2\pi)$, $u_k \in \mathcal{F}_{\theta}(\bar{Y})$
    with $|u_k|=1$ and $|u_k|_{\theta} \leq C |b_k|$, such that setting $n_k = [\zeta \ln|b_k|]$, $$
    \left\vert (M^{n_k}_{b_k}u_k)(y)) - e^{i \psi_k} u_k(y) \right\vert \leq C_6 |b_k|^{-\alpha}
    $$
    for all $y \in Z_0$, $k \geq 1$. $\mathcal{F}_{\theta}$ and $|\cdot|_{\theta}$ are a Banach space and its associated norm defined in \cite{Balint2019polynomial}.
\end{dfnt}

There are different strategies to show the absence of approximate eigenfunctions.
One is to consider the \emph{temporal distance} defines as follows.
Let $y_1,y_4 \in Y$ and set $y_2 = W^s(y_1) \cap W^u(y_4)$ and $y_3 = W^u(y_1) \cap W^s(y_4)$ then the temporal distance is defined as $$
D(y_1,y_4)= \sum_{n= - \infty}^{\infty} \phi(F^n(y_1))-\phi(F^n(y_2))+\phi(F^n(y_3))-\phi(F^n(y_4)).
$$
\begin{lem}[Theorem 5.6 of \cite{melbourne2018superpolynomial}]
  Let $Z_0= \cap_{n=0}^\infty F^{-n} Z$ where $Z$ is a union of finitely many elements of the partition $\{ Y_j \}$. Let $Z_0$ denote the corresponding finite subsystem of $\bar{Y}$. If the lower box dimension of $D(Z_0 \times Z_0)$ is positive, then there do not exist approximate eigenfunctions on $\bar{Z_0}$.
\end{lem}

For flows with a contact structure, a formula for $D$ exists in \cite{KatokBurns} at Lemma 3.2., and then the lower box dimension is positive as in \cite{melbourne2009decay} Example 5.7.

As this last condition is verified in our case, the flow of billiard table, the absence of eigenfunctions condition is checked.

\subsection{Functional space}

Given $v : \QQ \times \mathbb{S}^1 \to \mathbb{R}$ and $\eta > 0$ we define $$
|v|_{\eta}= \sup_{x \ne x'}\frac{|v(x)-v(x')|}{d_{\QQ \times \mathbb{S}^1}(x,x')^{\eta}}.
$$
Then $\|v \|_{\eta}= |v|_{\eta} + \| v \|_{\infty}$. We call $\mathcal{C}^{\eta}$ the space of function with finite $\| \cdot \|_{\eta}$ norms.

Then we will need to control regularity in the flow direction. For $v$ a function, let's define $$
|v|_{\eta,0}= \sup_{x \in \QQ \times \mathbb{S}^1 , t>0} \frac{v(\FF^t x)-v(x)}{t^{\eta}}.
$$
Again $\| \cdot \|_{\eta,0}= |\cdot|_{\eta,0}+ \| \cdot \|_{\infty}$ and $\mathcal{C}^{\eta,0}$ is the corresponding functional space.

The \emph{differentiability in the flow direction} a function $v$ is defined as the well-definess of $$
\partial_{\FF} v := \lim_{t \to 0} \frac{v(\FF^t x)-v(x)}{t}.
$$
For a $m$-differentiable function in the flow direction define $ \| v \|_{\eta,m}=\sum_{j=0}^m \| \partial_{\FF}^m v\|_{\eta}$ and call $\mathcal{C}^{\eta,m}$ the associated functional space.

In our case we can take in Theorem \ref{thm_vitesse_billiard_continue} the norm to be $\| \|_{\eta,m}$ with any $\eta \in (0,1]$ and $m$ given by Theorem \ref{thm_theorique_Gibbs_Markov_map}.

\subsection{What we need to show}

To sum up this section, we emphasize that the last condition that we need to check to apply Theorem \ref{thm_theorique_Gibbs_Markov_map} in our case is the following.

There is a uniformly hyperbolic part of $\MM$ called $\Sigma$, with $T_{\Sigma}$ the return map to $\Sigma$ and $\Theta$ the associated roof function. On it, we have the following, there is $\gamma \in (0,1)$ such that for $y,y'\in \Sigma$,

\begin{itemize}
  \item $|\Theta(y)-\Theta(y')| \leq C_4 d_{\MM}(y,y')^{\gamma}$ if $y \in W^s(y')$,
  \item $|\Theta(y)-\Theta(y')| \leq C_4 d_{\MM}(T_{\Sigma}(y),T_{\Sigma}(y'))^{\gamma}$ if $y \in W^u(y')$,
\end{itemize}

These Hölder conditions are the object of the following section. Once it has been done, we would have demonstrated Theorem \ref{thm_vitesse_billiard_continue} using Theorem \ref{thm_theorique_Gibbs_Markov_map} where the power $\delta$ is equal to $a+1$ and one can take the class of function which are $\gamma$-Hölder and $m,\gamma$-Hölder in the flow direction as the required regularity.

\section{Computations for the Lipschitz condition}
\label{sec_calcul_billiard}

\subsection{Definition and structure of $\Sigma$}

Let $\epsilon >0$ be a fixed, small enough constant fixed until the end of the article.
Following Chernov and Zhang \cite{chernov2005family}, we define a part of the billiard away from the points of vanishing curvature.

$$
\Sigma = \partial Q \backslash \{ |x| < \epsilon \} \times [-\pi /2, \pi /2].
$$

We call $\tau$ the return time from a point of $ \MM $ to $\MM$ and $\Theta$ the return time from $\Sigma$ to $\Sigma$. Moreover we call $T_{\Sigma}$ the return map to $\Sigma$.

Let's recall from \cite{chernov2005family} the structure of $\Sigma$ near $\{ |x| < \epsilon \} \subset \partial Q$.
Trajectories leaving $\Sigma$ and such that the first rebound is in $\mathcal{P}:= \{ |x| < \epsilon \} \times [-\pi /2, \pi /2]$ have typically two behaviors. The first one move through the window. This set is decomposed by the number of rebound inside $\mathcal{P}$, and we will call $C^n$ the set for where it bounces $n$ times.

A second type of trajectory approach the central axis and then turn back and hit $\Sigma$ on the same side. As previously, we can refine this set by counting the number of bounces and call $C^{-n}$ when it does $n$ bounces.

\begin{figure}[!h]
\centering

\tikzset{every picture/.style={line width=0.75pt}} 

\tikzset{every picture/.style={line width=0.75pt}} 

\tikzset{every picture/.style={line width=0.75pt}} 

\begin{tikzpicture}[x=0.75pt,y=0.75pt,yscale=-1,xscale=1]

\draw [line width=2.25]    (74,255) .. controls (94.32,221.39) and (196.31,203.38) .. (301.5,202.87) .. controls (405.98,202.36) and (513.61,219.12) .. (547.5,255) ;
\draw [line width=2.25]    (74,55) .. controls (115.5,113) and (204.31,109.51) .. (309.5,109) .. controls (414.69,108.49) and (505.5,112) .. (547.5,55) ;
\draw  [dash pattern={on 0.84pt off 2.51pt}]  (320,109) -- (320,203) ;
\draw  [dash pattern={on 0.84pt off 2.51pt}]  (432,106) -- (432,212) ;
\draw  [dash pattern={on 0.84pt off 2.51pt}]  (210,107.67) -- (210,208.67) ;
\draw [color={rgb, 255:red, 74; green, 144; blue, 226 }  ,draw opacity=1 ] [dash pattern={on 4.5pt off 4.5pt}]  (418.33,106.67) -- (470.25,219.5) ;
\draw [shift={(446.38,167.63)}, rotate = 245.29] [fill={rgb, 255:red, 74; green, 144; blue, 226 }  ,fill opacity=1 ][line width=0.08]  [draw opacity=0] (10.72,-5.15) -- (0,0) -- (10.72,5.15) -- (7.12,0) -- cycle    ;
\draw [color={rgb, 255:red, 74; green, 144; blue, 226 }  ,draw opacity=1 ] [dash pattern={on 4.5pt off 4.5pt}]  (418.33,106.67) -- (389.67,206.67) ;
\draw [color={rgb, 255:red, 74; green, 144; blue, 226 }  ,draw opacity=1 ] [dash pattern={on 4.5pt off 4.5pt}]  (375.67,109.67) -- (389.67,206.67) ;
\draw [color={rgb, 255:red, 74; green, 144; blue, 226 }  ,draw opacity=1 ] [dash pattern={on 4.5pt off 4.5pt}]  (375.67,108.67) -- (365,203.67) ;
\draw [color={rgb, 255:red, 74; green, 144; blue, 226 }  ,draw opacity=1 ] [dash pattern={on 4.5pt off 4.5pt}]  (367,107.67) -- (365,204.67) ;
\draw [shift={(365.9,161.17)}, rotate = 271.18] [fill={rgb, 255:red, 74; green, 144; blue, 226 }  ,fill opacity=1 ][line width=0.08]  [draw opacity=0] (10.72,-5.15) -- (0,0) -- (10.72,5.15) -- (7.12,0) -- cycle    ;
\draw [color={rgb, 255:red, 74; green, 144; blue, 226 }  ,draw opacity=1 ] [dash pattern={on 4.5pt off 4.5pt}]  (367,107.67) -- (369,203.33) ;
\draw [color={rgb, 255:red, 74; green, 144; blue, 226 }  ,draw opacity=1 ] [dash pattern={on 4.5pt off 4.5pt}]  (386,109.25) -- (369,205.33) ;
\draw [color={rgb, 255:red, 74; green, 144; blue, 226 }  ,draw opacity=1 ] [dash pattern={on 4.5pt off 4.5pt}]  (386,108.25) -- (402.25,207.5) ;
\draw [color={rgb, 255:red, 74; green, 144; blue, 226 }  ,draw opacity=1 ] [dash pattern={on 4.5pt off 4.5pt}]  (442,110.25) -- (402.25,207.5) ;
\draw [shift={(420.23,163.5)}, rotate = 292.23] [fill={rgb, 255:red, 74; green, 144; blue, 226 }  ,fill opacity=1 ][line width=0.08]  [draw opacity=0] (10.72,-5.15) -- (0,0) -- (10.72,5.15) -- (7.12,0) -- cycle    ;
\draw [color={rgb, 255:red, 208; green, 2; blue, 27 }  ,draw opacity=1 ]   (138.5,219) -- (195.8,166.04) ;
\draw [shift={(198,164)}, rotate = 137.25] [fill={rgb, 255:red, 208; green, 2; blue, 27 }  ,fill opacity=1 ][line width=0.08]  [draw opacity=0] (10.72,-5.15) -- (0,0) -- (10.72,5.15) -- (7.12,0) -- cycle    ;
\draw [color={rgb, 255:red, 208; green, 2; blue, 27 }  ,draw opacity=1 ]   (198,164) -- (257.5,109) ;
\draw [color={rgb, 255:red, 208; green, 2; blue, 27 }  ,draw opacity=1 ]   (254.24,109.76) -- (311.04,155.98) ;
\draw [shift={(313.37,157.88)}, rotate = 219.14] [fill={rgb, 255:red, 208; green, 2; blue, 27 }  ,fill opacity=1 ][line width=0.08]  [draw opacity=0] (10.72,-5.15) -- (0,0) -- (10.72,5.15) -- (7.12,0) -- cycle    ;
\draw [color={rgb, 255:red, 208; green, 2; blue, 27 }  ,draw opacity=1 ]   (313.37,157.88) -- (372.5,206) ;
\draw [color={rgb, 255:red, 208; green, 2; blue, 27 }  ,draw opacity=1 ]   (368.73,205.38) -- (410.18,156.48) ;
\draw [shift={(412.12,154.19)}, rotate = 130.28] [fill={rgb, 255:red, 208; green, 2; blue, 27 }  ,fill opacity=1 ][line width=0.08]  [draw opacity=0] (10.72,-5.15) -- (0,0) -- (10.72,5.15) -- (7.12,0) -- cycle    ;
\draw [color={rgb, 255:red, 208; green, 2; blue, 27 }  ,draw opacity=1 ]   (412.12,154.19) -- (455.5,103) ;

\draw (311.33,87.73) node [anchor=north west][inner sep=0.75pt]    {$x_{0}$};
\draw (311.33,207.07) node [anchor=north west][inner sep=0.75pt]    {$x_{1}$};
\draw (424,219.4) node [anchor=north west][inner sep=0.75pt]    {$\epsilon $};
\draw (206,216.4) node [anchor=north west][inner sep=0.75pt]    {$-\epsilon $};
\draw (132,226.4) node [anchor=north west][inner sep=0.75pt]    {$p$};
\draw (466,224.4) node [anchor=north west][inner sep=0.75pt]    {$q$};

\end{tikzpicture}

\caption{Two different kind of trajectory}
\end{figure}
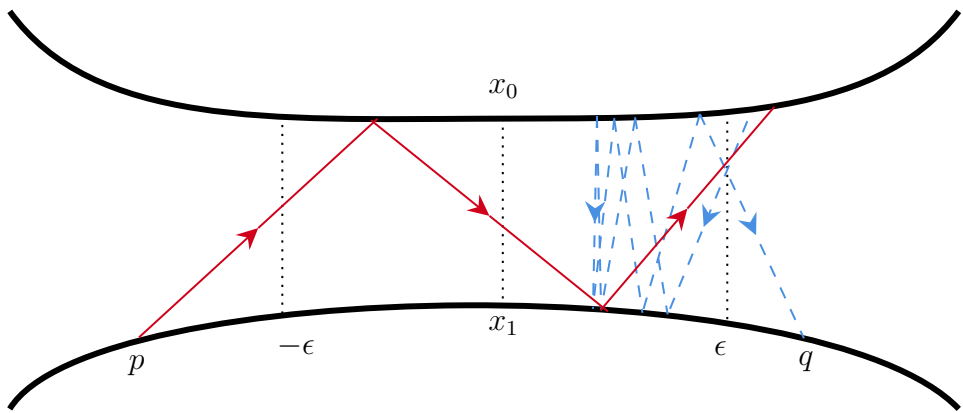

In the following of the paper, when we talk about a point $(x,\phi) \in \MM$, the measure of the angle $\phi$ is taken to be the measure of the angle between the trajectory and the vertical direction, not the normal vector to $\MM$.

The dimension of the stable and unstable variety inside $C^n$ for $n \in \mathbb{Z}$ is as follows.

\begin{lem}
  \label{lem_structure_Sigma}
  Let $p=(x,\phi),q=(y,\psi) \in C^n$ such that $q \in W^s(p)$, then
  $$
|\phi-\psi| = O(n^{-b})
  $$
  where $b=2+\frac{\beta+2}{\beta-2}$.
  If $q \in W^u(p)$, then
  $$
  |\phi-\psi| = O(n^{-b}) \text{ and } |x-y| = O(n^{-b})
  $$
\end{lem}

It will be more convenient for us to work with the folding of the table along the $x$-axis. In this framework, $C^n$ and $C^{-n}$ are defined by counting the number of bounces along the curve $-g_{\beta}$ (and not the number of bouncing along the $x$-axis).

\subsection{$\Theta$ is Lipschitz}

 We begin this section by giving two lemmas which are true for the two kind of trajectories.

First, we could compute the return time to the border $\tau$ in the window $\PP$ in the natural coordinate.
\begin{lem}
  \begin{equation}
    \label{eq_return_time}
    \tau(x,\phi)=\frac{1+x^{\beta}}{\cos \phi} + \frac{1+x'^{\beta}}{\cos \phi}
  \end{equation}

  where $x'$ is the position of the next point.
\end{lem}
\begin{proof}
The proof is direct with trigonometric considerations.
\end{proof}
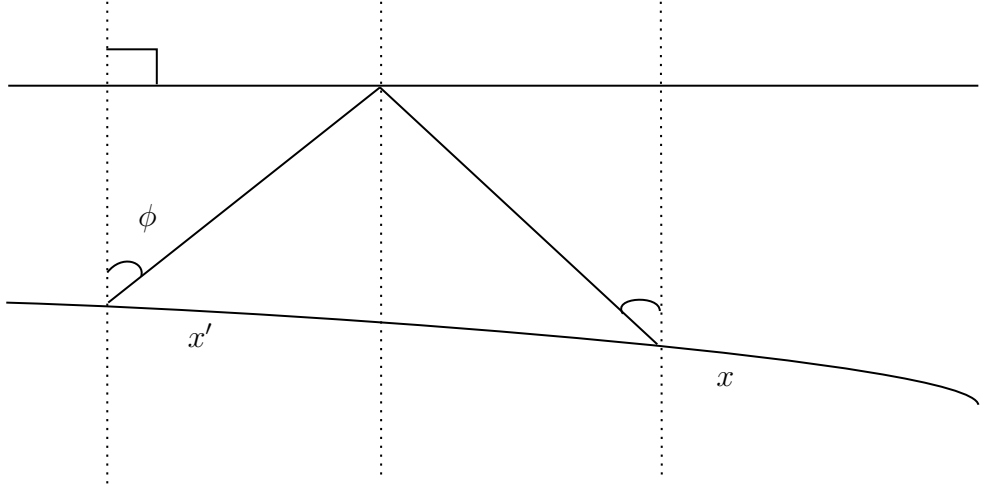
\begin{figure}[!h]
\centering

\tikzset{every picture/.style={line width=0.75pt}} 

\begin{tikzpicture}[x=0.75pt,y=0.75pt,yscale=-1,xscale=1]

\draw    (41.47,58.24) -- (525.5,58.24) ;
\draw    (40.5,167.36) .. controls (207.17,172.29) and (522.59,197.64) .. (525.5,218.76) ;
\draw  [dash pattern={on 0.84pt off 2.51pt}]  (227.52,16) -- (227.52,254.66) ;
\draw    (226.94,59.09) -- (365.32,188.34) ;
\draw    (226.94,59.09) -- (91.37,167.36) ;
\draw  [dash pattern={on 0.84pt off 2.51pt}]  (367.06,16) -- (367.55,253.96) ;
\draw  [dash pattern={on 0.84pt off 2.51pt}]  (90.89,16) -- (90.89,261) ;
\draw    (90.89,152.23) .. controls (98.16,142.72) and (109.79,146.95) .. (107.85,153.99) ;
\draw    (348.17,173) .. controls (342.35,164.55) and (366.58,163.14) .. (366.58,171.59) ;
\draw   (90.41,39.94) -- (115.6,39.94) -- (115.6,57.54) ;

\draw (104.74,116.12) node [anchor=north west][inner sep=0.75pt]    {$\phi $};
\draw (393.39,201.01) node [anchor=north west][inner sep=0.75pt]    {$x$};
\draw (129.59,176.37) node [anchor=north west][inner sep=0.75pt]    {$x'$};

\end{tikzpicture}
\caption{Length between two bounces on the lower part}
\end{figure}

To control the return times in equation \ref{eq_return_time}, we have to control the denominator. The following lemma assure that the angles are such that their cosines are bounded away from $0$.

\begin{lem}
\label{lem_angle_pas_trop_petit}
There is a constant $\Phi$ such that for every $|n| \geq 2$, every $1 \leq m \leq n$ and every $p=(x,\phi)\in C^n$ we have
$$
|\phi_m - \pi/2| > \Phi.
$$
Consequently $cos(\phi_m)$ is bounded away from $0$.
\end{lem}

\begin{proof}
If $n \geq 2$, the trajectory contains at least two bounces. If we named $x$ and $y$ the two most far apart then their distance is at most $2 \epsilon$, and  $\cos(\phi)$ is at least $\frac{|\epsilon|^{\beta}}{\sqrt{\epsilon^2+(|\epsilon|^{\beta})^2}}$.
\end{proof}

Now, we will focus on the trajectories going through the window. So let $n>0$.

To control the trajectories we will cut them in three pieces:
\begin{itemize}
  \item the first part where the angles are still large,
  \item a second part until we cross the line $x=0$,
  \item the rest of the trajectory until we come back to the hyperbolic part.
\end{itemize}

Let $p=(x,\phi) \in C^n$. We define $1 \leq n' \leq n$ to be the integer such that $x_{n'+1} \leq 0 \leq x_{n'}$ and $n'' \in [1,n']$ be uniquely defined by $$
w_{n''-1} > 2 w_{n'} > w_{n''}.
$$

Then we have the different estimates
\begin{lem}
  \label{lem_estim_billiard}

  In the first part of the trajectory, we have for all $1<m\leq n''$
  $$
  x_m \sim m^{\frac{2}{2-\beta}}.
  $$

  In the second part, for all $n'' \leq m \leq n'$
  $$
    x_m \sim (n'-m)w_n'
  $$
  and
  $$
  n^{\frac{\beta}{2-\beta}} \sim w_{n'}
  $$

\end{lem}

Furthermore, we will need to control the sum of power of position of the bounces along a trajectory.

\begin{lem}
  \label{lem_sum_pos_beta}
  Let $p=(x,\phi) \in C^n$. If $x_m$ is the position after the $m$-th rebound then $$
  \sum_{m=0}^n x_m^{\beta-1} \leq C_{Lem_{3.5}}
  $$
  where $C_{Lem_{3.5}}$ does not depend on $n$.
\end{lem}

\begin{proof}
  We cut the sum in three parts
  $$
  \sum_{m=0}^n x_m^{\beta-1} = \sum_{m=0}^{n''} x_m^{\beta-1} + \sum_{m=n''+1}^{n'} x_m^{\beta-1} + \sum_{m=n'}^n x_m^{\beta-1}.
  $$
  For the first part, lemma \ref{lem_estim_billiard} gives that there exist a $C_{Lem_{3.4}}$ such that $x_m \leq C_{Lem_{3.4}} m^{\frac{2}{2-\beta}}$, so
  $$
  \sum_{m=0}^{n''} x_m^{\beta-1} \leq C_{Lem_{3.4}} \sum_{m=1}^{n''} m^{\frac{2(\beta-1)}{2-\beta}} < + \infty.
  $$
  For the second part, the same lemma \ref{lem_estim_billiard} indicates that there exists $C_{Lem_{3.4}}'$ such that $x_m \leq C_{Lem_{3.4}}'(n'-m)n^{\frac{\beta}{\beta-2}}$, hence
  \begin{align*}
      \sum_{m=n''+1}^{n'} x_m^{\beta-1} & \leq C_{Lem_{3.4}}' \sum_{m=n''+1}^{n'} (n'-m)^{\beta-1}n^{\frac{\beta(\beta-1)}{2-\beta}} \\
      & \leq C_{Lem_{3.4}}' \sum_{m=1}^{n} n^{\beta-1}n^{\frac{\beta(\beta-1)}{2-\beta}} \\
      & \leq C_{Lem_{3.4}}' n^{\beta+\frac{\beta(\beta-1)}{2-\beta}} \ \\
      & = C_{Lem_{3.4}}' n^{-\frac{\beta}{\beta-2}}.
  \end{align*}
  As the last term goes to $0$ when $n$ goes to infinity, the sequence is bounded.

  Finally, for the third term, by reversing time, one can consider a trajectory coming from the hyperbolic part where $x < 0$, and bound it as the two first terms above.
\end{proof}

We can now demonstrate that the return time to the hyperbolic part is Hölder.

\begin{lem}
  There is a $\gamma \in (0,1)$ such that for any $p=(x,\phi) \in C^n$ and $q=(y,\psi) \in C^n \cap W^s(x)$,
  $$
  \vert \Theta(p)-\Theta(q) \vert \leq d_{\MM}(p,q)^\gamma.
  $$
\end{lem}

\begin{proof}
  We will write $p_m=(x_m,\phi_m)$ and $q_m=(y_m,\psi_m)$ for the $m$-th rebound.

  Without loss of generality we can consider that $x<y$.

  This implies that $x_m \leq y_m$ for any $m \in [1,n]$ and that $\phi_m \leq \psi_m$.

With the triangle inequality, we bound the difference between the return times to the hyperbolic part by the difference of return times for each return to the boundary of the billiard.

   $$
  |\Theta(p)-\Theta(q)| \leq \sum_{m=0}^n |\tau(p_m) - \tau(q_m)|.
  $$

  Then

  \begin{align*}
    \tau(p_m) - \tau(q_m) & = \frac{1+x_m^{\beta}}{\cos (\phi_m)} + \frac{1+x_{m+1}^{\beta}}{\cos (\phi_m)} -\frac{1+y_m^{\beta}}{\cos (\psi_m)} - \frac{1+y_{m+1}^{\beta}}{\cos (\psi_m)} \\
    & = \frac{1}{cos \phi_m}(x_m^{\beta}-y_m^{\beta}+x_{m+1}^{\beta}-y_{m+1}^{\beta})+(\frac{1}{\cos (\phi_m)}-\frac{1}{\cos \psi_m})(2+y_m^{\beta}+y_{m+1}^{\beta}).
  \end{align*}
  As lemma \ref{lem_angle_pas_trop_petit} indicates there is a $K>0$ such that $K < \cos \phi_m, \cos \psi_m$.
  Then
  \begin{align*}
    \left\vert \tau(p_m) - \tau(q_m) \right\vert & \leq \frac{1}{K}(|x_m^{\beta}-y_m^{\beta}|+|x_{m+1}^{\beta}-y_{m+1}^{\beta}|) + \frac{2+2 \epsilon^\beta}{K^2}|\cos \phi_m - \cos \psi_m|.
  \end{align*}
  using the mean value theorem there is $z_m \in [x_m,y_m]$ and $z_{m+1}\in [x_{m+1},y_{m+1}]$ such that
  \begin{align*}
    \left\vert \tau(p_m) - \tau(q_m) \right\vert & \leq \frac{1}{K}(\beta z_m^{\beta-1}|x_m-y_m|+\beta z_{m+1}^{\beta-1}|x_{m+1}-y_{m+1}|) + \frac{2+2 \epsilon^\beta}{K^2}|\cos \phi_m - \cos \psi_m|\\
  \end{align*}

  So

  \begin{align}
    \label{eq_control_diff_theta}
    |\Theta(p_0)-\Theta(q_0) | & \leq \sum_{m=0}^n |\tau(p_m) - \tau(q_m)| \nonumber \\
    & \leq \frac{2 \beta}{K} \sum_{m=1}^n z_m^{\beta-1}|x_m-y_m| + \sum_{m=0}^n \frac{2+2 \epsilon^\beta}{K^2}|\cos \phi_m - \cos \psi_m| \\
    & \leq \frac{2 \beta |x_0-y_0|}{K} \sum_{m=1}^n y_m^{\beta-1} + \frac{2+2 \epsilon^\beta}{K^2}n|\cos \phi_0 - \cos \psi_0|. \nonumber
  \end{align}
  To get the last inequality we use the fact that the differences of distance and angle along two trajectories in the same stable variety decrease.
Using lemma \ref{lem_sum_pos_beta} we can bound the sum $\sum_{m=1}^n y_m^{\beta-1}$. Moreover, using Lemma \ref{lem_structure_Sigma}, we have that
\begin{align*}
  n |\cos \phi_0 - \cos \psi_0| & \leq n |\phi_0 - \psi_0| \\
  & = n |\phi_0 - \psi_0|^{\gamma}|\phi_0 - \psi_0|^{(1-\gamma)}\\
  & \leq n |\phi_0 - \psi_0|^{\gamma}n^{-b(1-\gamma)} \\
  & = |\phi_0 - \psi_0|^{\gamma} n^{1-b(1-\gamma)}
\end{align*}
where $b=2 + \frac{\beta+2}{\beta-2} \geq 1$.
Taking $\gamma \in [0,1-1/b]$, yields a negative exponent.

Putting everything together, we have shown that$$
  |\Theta(p_0)-\Theta(q_0) | \leq |x_0- y_0| + |\phi_0 - \psi_0|^{\gamma}
$$

As we have used lemma \ref{lem_angle_pas_trop_petit} which is true for $n \geq 2$, we might also include estimation of the $\gamma$-Hölder constant of $\Theta$ in $C^1$.
\end{proof}

We can demonstrate the same lemma for unstable manifold.

\begin{lem}
  There is a $\gamma \in (0,1)$ such that for any $p=(x,\phi) \in C^n$ and $q=(y,\psi) \in C^n \cap W^u(x)$,
  $$
  \vert \Theta(p)-\Theta(q) \vert \leq d_{\MM}(T_{\Sigma}(p),T_{\Sigma}(q))^\gamma.
  $$
\end{lem}

\begin{proof}
  The proof the follow the same path as in the stable manifold case. Indeed, one just need to change the passage from the second line of equation \ref{eq_control_diff_theta} but instead of majoring each $|x_m-y_m|$ by $|x_0-y_0|$, we bound them with $|x_n-y_n|$ and replace accordingly the upper bound for the angles.
\end{proof}

Finally, for trajectories which leave the window by the same side that they went in, so $n<0$, one can define $n'$ to be the time when the trajectory changes its horizontal direction so $x_{n'-1} > x_{n'}$ and $x_{n'+ 1} > x_{n'}$. The same estimates hold for $1 \leq m \leq n''$ as in the first case and the second part can be dealt with the same trick.

We have now verified every hypothesis of Theorem \ref{thm_theorique_Gibbs_Markov_map}.
In our case $M$ is $\QQ \times \mathbb{S}^1$ and we work with the Poincaré section given by the return map to $\MM$, which is $X$ in theorem \ref{thm_theorique_Gibbs_Markov_map}. The roof function $h$ is our $\tau$. We easily chek that $\tau$ is always striclty positive and have an upper bound. The hyperbolic part $Y$ is given by $\Sigma$.
We can deduce our main Theorem \ref{thm_vitesse_billiard_continue}.


\bibliographystyle{plainnat}
\bibliography{bibliographie.bib}

\end{document}